\documentclass[english]{article}
\usepackage[T1]{fontenc}
\usepackage[latin9]{inputenc}
\usepackage{amsmath}
\usepackage{amsthm}
\usepackage{amssymb}
\usepackage{esint}

\makeatletter
\theoremstyle{plain}
\newtheorem{thm}{\protect\theoremname}
  \theoremstyle{definition}
  \newtheorem{defn}[thm]{\protect\definitionname}
  \theoremstyle{plain}
  \newtheorem{prop}[thm]{\protect\propositionname}
  \theoremstyle{remark}
  \newtheorem{rem}[thm]{\protect\remarkname}

\usepackage{babel}

\usepackage{color}

\def\O{\Omega}

\def\n{\nabla}

\def\p{\partial}

\def\a{\alpha}
\def\b{\beta}
\def\n{\nabla}

\def\O{\Omega}
\def\p{\partial}

\def\a{\alpha}
\def\b{\beta}

\def\d{\delta}

\def\l{\lambda}
\def\s{\sigma}

\def\De{\Delta}
\def\n{\nabla}
\def\<{\langle}
\def\>{\rangle}

\def\De{\Delta}
\def\n{\nabla}

\def\O{\Omega}
\def\p{\partial}

\def\a{\alpha}
\def\b{\beta}

\def\d{\delta}
\def\l{\lambda}
\def\s{\sigma}

\def\ep{\epsilon}

\usepackage{babel}
\providecommand{\definitionname}{Definition}
  \providecommand{\propositionname}{Proposition}
  \providecommand{\remarkname}{Remark}
\providecommand{\theoremname}{Theorem}

\makeatother

\usepackage{babel}
  \providecommand{\definitionname}{Definition}
  \providecommand{\propositionname}{Proposition}
  \providecommand{\remarkname}{Remark}
\providecommand{\theoremname}{Theorem}

\begin{document}

\title{Classical Neumann Problems for Hessian equations and Alexandrov-Fenchel's inequalities}

\author{Guohuan Qiu \thanks{Department of Mathematics, University of Science and Technology of
China, Hefei , P. R. China \quad{}Email: guohuan@ustc.edu.cn. The
research of GQ was supported by China Scholarship Council. } \and Chao Xia \thanks{School of Mathematical Sciences, Xiamen University, 361005, Xiamen,
China\quad{}Email: chaoxia@xmu.edu.cn. Research of CX is supported
in part by the Fundamental Research Funds for the Central Universities
(Grant No. 20720150012), NSFC (Grant No. 11501480).}}
\maketitle
\begin{abstract}
Recently, the first named author together with Xinan Ma \cite{ma2015neumann},
have proved the existence of the Neumann problems for Hessian equations.
In this paper, we proceed further to study classical Neumann problems
for Hessian equations. We prove here the existence of classical
Neumann problems under the uniformly convex domain in $\mathbb{R}^{n}$.
As an application, we use the solution of the classical Neumann problem
to give a new proof of a family of Alexandrov-Fenchel inequalities
arising from convex geometry. This geometric application is motivated
from Reilly \cite{Reilly1980}. 
\end{abstract}

\section{Introduction}

Let $\O$ be a bounded domain in $\mathbb{R}^{n}$ with smooth boundary
$\p\O$. It is well-known that for two sufficiently regular function
$f$ on $\bar{\Omega}$ and $\varphi$ on $\p\O$, the following classical
Neumann boundary value problem for Poisson's equation: 
\begin{equation}
\begin{cases}
\De u & =f(x)\begin{array}{llll}
 & \hbox{ in }\Omega,\end{array}\\
u_{\nu} & =\varphi(x)\begin{array}{llll}
 & \hbox{ on }\partial\Omega\end{array}
\end{cases}\label{eq1}
\end{equation}
admits a classical solution $u$, which is unique up to an additive
constant, if and only if 
\begin{equation}
\int_{\Omega}fdx=\int_{\partial\Omega}\varphi d\mu.\label{eq:compatibility}
\end{equation}
Here $\nu$ is the outward unit normal of $\p\O$ and $u_{\nu}=\frac{\p u}{\p\nu}$
is the normal derivative of $\nu$. This result was proved by using
the Fredholm alternative from functional analysis, See e.g. \cite{gilbarg2015elliptic},
pp. 130. See also a recent paper by Nardi \cite{Nardi}.

It is natural to ask whether similar result holds for $k$-Hessian
equation $\sigma_{k}(D^{2}u)=f$. The Neumann type problems for Monge-Ampere
type equations have been well studied by Lions-Trudinger-Urbas \cite{lions1986neumann}.
Trudinger \cite{trudinger1987degenerate} considered and proved the
existence for the Neumann type problem for $k$-Hessian equation in
the case when the domain is a ball and he conjectured similar result
holds for general uniformly convex domains. In a recent paper \cite{ma2015neumann},
the fist named author and Ma gave an affirmative answer to Trudinger's
conjecture. Particularly, the Neumann type boundary condition in \cite{ma2015neumann}
is $u_{\nu}=\varphi(x,u)$, where $\frac{\partial\varphi}{\partial u}\leq-c_{0}<0$
for some positive $c_{0}$. %e existence for the following type Neumann problem for Hessian equations: %\begin{equation}\label{eq1}%\begin{cases}%\sigma_{k}(D^{2}u)&=f(x,u)  \begin{array}{cc}%&\hbox{in }\Omega\end{array},\\%\frac{\p u}{\p \nu}&=\varphi(x, u)  \begin{array}{cc}%&\hbox{on }\partial\Omega\end{array},%\end{cases}%\end{equation}%when $\O$ is a uniformly convex domain. Here $\sigma_k(D^2 u)$ is the $k$-th elementary symmetric function $\sigma_k$ acting on the eigenvalues of the Hessian matrix  of $u$. %The assumption on $\varphi(x,u)$ is. This result gave an affimative answer to a question of Trudinger \cite{Tr}.

It is clear that the assumption on $\varphi$ excludes the case that
$\varphi(x,u)$ only depends on $x$, i.e., $\varphi(x,u)=\varphi(x)$.
In this paper, we will study this case, namely, the classical Neumann
boundary value problem for the $k$-Hessian equation: 
\begin{equation}
\begin{cases}
\sigma_{k}(D^{2}u)=f(x) & \begin{array}{cc}
 & \hbox{ in }\Omega,\end{array}\\
u_{\nu}=\varphi(x) & \begin{array}{cc}
 & \hbox{ on }\partial\Omega.\end{array}
\end{cases}\label{eq2}
\end{equation}

%where $\sigma_{k}(D^{2}u)$ is the $k$-th elementary symmetric function%of eigenvalues of $D^{2}u$, the Hessian matrix of $u$,  $\nu$ is the outer unit normal vector fieldwhere%$f$ and $\varphi$ are sufficient regular functions on $\bar{\O}$.%When $k=1$, \eqref{eq1} is the Laplace equation%with a Neumann boundary condition. 
It turns out that the existence may not hold for general $f(x)$ and
$\varphi(x)$. This is because they should satisfy some compatibility
condition as \eqref{eq:compatibility}. In the case $k=n$, Lions-Trudinger-Urbas
\cite{lions1986neumann} showed that for sufficient regular $f$ and
$\varphi$ on $\bar{\O}$ with $f>0$, there exists a pair $(\lambda,u)$
satisfying 
\begin{equation}
\begin{cases}
{\rm det}(D^{2}u)=f(x) & \begin{array}{cc}
 & \hbox{ in }\Omega,\end{array}\\
u_{\nu}=\lambda+\varphi(x) & \begin{array}{cc}
 & \hbox{ on }\partial\Omega.\end{array}
\end{cases}
\end{equation}
Here $\lambda$ is a unique constant while $u$ is unique up to an
additive constant.

Our main result in this paper is the following 
\begin{thm}
\label{thm:Classic} Let $\Omega$ be a $C^{4}$ bounded, uniformly
convex domain in $\mathbb{R}^{n}$. Let $f\in C^{2}(\overline{\Omega})$
with $f>0$ and $\varphi\in C^{3}(\overline{\Omega})$. Then there
is a unique constant $\lambda$ and a unique $k$ admissible solution
$u\in C^{3,\alpha}(\overline{\Omega})$ up to an additive constant,
satisfying 
\begin{equation}
\begin{cases}
\sigma_{k}(D^{2}u)=f(x) & \begin{array}{cc}
 & \hbox{ in }\Omega,\end{array}\\
u_{\nu}=\lambda+\varphi(x) & \begin{array}{cc}
 & \hbox{ on }\partial\Omega.\end{array}
\end{cases}\label{eq:classic}
\end{equation}

\end{thm}
A solution $u$ is called $k$-admissible if the eigenvalue of the
Hessian matrix $D^{2}u$ belongs to $\Gamma_{k}^{+}$, see Section
2.

%But they did not give a geometric interpretation of this constant%$\lambda$.
 We remark that unlike the case $k=1$, we have no explicit expression
for $\lambda$. However, it is easy to see a lower bound for $\l$:
\begin{eqnarray}
\l{\rm Area}(\p\O)\geq n\int_{\O}\left(\frac{f(x)}{C_{n}^{k}}\right)^{\frac{1}{k}}dx-\int_{\p\O}\varphi(x)d\mu.\label{lambda}
\end{eqnarray}
%Indeed, by Newton-Maclaurin's inequality and integration by parts we have that%\begin{eqnarray*}%\int_{\p\O} \l+\varphi(x)=\int_{\p\O} \De u dx\geq n\int_{\O} \left(\frac{\s_k(D^2 u)}{C_n^k}\right)^{\frac1k}dx=n\int_{\O} \left(\frac{f(x)}{C_n^k}\right)^{\frac1k}dx.%\end{eqnarray*}

Therefore, the classical Neumann problem for $k$-Hessian equations
\eqref{eq2} may have no solutions. For example, in the case that
$f=1$ and $\varphi=0$, $\l$ has to be a nonzero constant by virtue
of \eqref{lambda}.

%but it is closely related the geometric information of $\Omega$ as%the case $k=1$. 

Let us illustrate the idea of the proof of Theorem \ref{thm:Classic}.
On one hand, Fredholm alternative is not applicable on \eqref{eq:classic}
as in the classical Neumann problem for Poisson's equation \eqref{eq1},
for we deal with fully nonlinear partial differential equations. On
the other hand, it is impossible to get a uniform bound for the solutions
to (\ref{eq:classic}) since a solution plus any constant is still
a solution. Thus we can not use continuity method to get the existence.
In order to overcome this difficulty we use a perturbation argument.
We first consider for any $\epsilon>0$ the following boundary value
problem 
\begin{equation}
\begin{cases}
\sigma_{k}(D^{2}u)=f(x) & \begin{array}{cc}
 & \hbox{ in }\Omega,\end{array}\\
u_{\nu}=-\epsilon u+\varphi(x) & \begin{array}{cc}
 & \hbox{ on }\partial\Omega.\end{array}
\end{cases}\label{eq:eps}
\end{equation}
The result in \cite{ma2015neumann} gives us the existence of $u_{\epsilon}$
for (\ref{eq:eps}). We then prove a gradient estimate for $u_{\epsilon}$
independent of $\epsilon$. Once we get this, we shall have all the
regularity estimate and by letting $\epsilon\rightarrow0$ we obtain
a solution of (\ref{eq:classic}). This kind of argument has been
used in \cite{lions1986neumann} when $k=n$. Their a priori gradient
estimate heavily depends on the convexity of the solutions. However,
in general $k<n$ case we have no convexity. Instead we use directly
maximum principle on some good choice of test functions. The choice
of test functions is motivated from \cite{MQX}, see also \cite{ma2015neumann}.

%It is well-known that if $f$ and $\varphi$ satisfies%the following compatible condition:%\begin{equation}%\int_{\Omega}fdx=\int_{\partial\Omega}\varphi d\mu,%\end{equation}%then there exists a unique classical solution, up to a constant multiplier, to \eqref{eq1} when $k=1$. See e.g. \cite{GT}, pp. 130. 

\
 A motivation to study such classical Neumann problem for $k$-Hessian
equations is to prove the Alexandrov-Fenchel inequalities \cite{alexandroff1937theorie}
which are of fundamental importance in the theory of convex geometry.
For an introduction of Alexandrov-Fenchel inequalities in convex geometry
we refer to Schneider's encyclopedia book \cite{schneider2013convex}.
In the past several decades, many mathematicians study such inequalities
from the viewpoint of PDEs. Particularly, Trudinger \cite{trudinger1994isoperimetric}
tried to use the Dirichlet problems for $k$-Hessian equations and
$k$-curvature equations to reach the following special cases of Alexandrov-Fenchel
inequalities between two quermassintegrals for a bounded (possibly
non-convex) domain $\O\subset\mathbb{R}^{n}$: 
\begin{equation}
\left(\frac{V_{n-1-k}(\Omega)}{V_{n-1-k}(B)}\right)^{\frac{1}{n-1-k}}\geq\left(\frac{V_{n-1-l}(\Omega)}{V_{n-1-l}(B)}\right)^{\frac{1}{n-1-l}},\quad n-1\ge k>l\ge-1,\label{eq:AF0}
\end{equation}
where $B$ is the unit ball in $\mathbb{R}^{n}$, and $V_{n-1-k}(\O)$
is the $k$-th quermassintegrals of $\O$. %V_{n}(\Omega)=vol(\Omega)$,%$V_{n-k-1}(\Omega):=\int_{\partial\Omega}\sigma_{k}(h)d\mu$, for%$k=0,\cdots,n-1$, where $d\mu$ is surface area of $\partial\Omega$,%and $h$ is the second fundamental form on $\partial\Omega$, and%$\sigma_{k}$ is defined in section 2. 
Guan and Li \cite{guan2009quermassintegral} used inverse mean curvature
type flow (parabolic PDEs) to prove inequalities \eqref{eq:AF0} for
star-shaped domains. Inspired by Gromov's proof \cite{Gromov} of
the isoperimetric inequality, Chang and Wang \cite{chang2013some}
established inequalities \eqref{eq:AF0} for $l=-1$ and $k=1,2$ when
$\Omega$ is $(k+1)$-convex by using a solution of some PDE from the
optimal transport, see also \cite{qiu2015family} for any $k$. They
\cite{chang2013} also proved \eqref{eq:AF0} for general $k$ for
$(k+1)$-convex domains with non-optimal constants. Cabr\'e \cite{Cabre}
used the Neummann problem for Possion's equation \eqref{eq1} with
$f=1$ and $\varphi=constant$ plus the Alexandrov-Bakelman-Pucci
(ABP) estimate to give a very simple proof of the classical isoperimetric
inequality in the Euclidean space.

The above results have a common feature that two geometric quantities
are compared. From the theory of convex geometry, we know that the
Alexandrov-Fenchel inequalities can link three quermassintegrals.
%But it is still unkown%whether three quantities version of Alexandrov-Fenchel inequalities links %are ture or not for nonconvex domain.
 By applying Reilly's formula \cite{Reilly1973,reilly1977applications}
on the solution of the classical Neumann problem for Possion's equation
\eqref{eq1} with $f=1$ and $\varphi=constant$, Reilly \cite{Reilly1980}
gave a new proof of the following Minkowski's inequality for convex
domains in the Euclidean space, 
\begin{eqnarray}
{\rm {Area}}(\p\O)^{2}\geq\frac{n}{n-1}{\rm {Vol}}(\O)\int_{\p\O}Hd\mu,\label{Mink}
\end{eqnarray}
where $H$ is the mean curvature of $\p\O$. A similar result has
been proved recently by the second author \cite{xia2015minkowski}
in the hyperbolic space.

In the same spirit, we can apply Reilly's high order formula on the
solution of the Neumann problem for the $k$-Hessian equation \eqref{eq:classic}
with $f=1$ and $\varphi=0$ to give a new proof of the following
special Alexandrov-Fenchel's inequalities for convex domains in $\mathbb{R}^{n}$. 
\begin{thm}
\label{thm:AF} Let $\Omega$ be a smooth bounded uniformly convex
domain in $\mathbb{R}^{n}$. For any $1\leq k\leq n-1$, denote by
$\s_{k}(h)$ the $k$-th mean curvature of $\partial\Omega$. Then

\begin{equation}
{\rm {Area}}(\p\O)^{k+1}\geq\frac{n^{k}}{C_{n-1}^{k}}{\rm {Vol}}(\Omega)^{k}\int_{\partial\Omega}\sigma_{k}(h)d\mu.\label{eq:AFIne}
\end{equation}
Equality holds if and only if $\Omega$ is a ball in $\mathbb{R}^{n}$. 
\end{thm}
In the case $k=n-1$, since $\int_{\partial\Omega}\sigma_{n-1}(h)d\mu$
is a dimensional constant, we get the isoperimetric inequality for
convex domains in $\mathbb{R}^{n}$. %In \cite{ma2015neumann}, the first author and Ma proved the following %\begin{thm}[\cite{ma2015neumann}]%Let $\Omega$ be a $C^{4}$ bounded uniformly convex domain in $\mathbb{R}^{n}$.%Where $f\in C^{2}(\overline{\Omega})$ is positive and $\varphi\in C^{3}(\overline{\Omega})$.%Then for any positive constant $\epsilon$, there exists a unique%$k$ admissible solution $u\in C^{3,\alpha}(\overline{\Omega})$ ,%%where $k$ admissible solution will be defined later in the following%section, of the boundary value problem, 

%\begin{equation}%\begin{cases}%\sigma_{k}(D^{2}u)=f(x) & \begin{array}{cc}%in & \Omega\end{array},\\%u_{\nu}=-\epsilon u+\varphi(x) & \begin{array}{cc}%%on & \partial\Omega.\end{array}%\end{cases}\label{eq:eps}%\end{equation}

%\end{thm}\

We remark that recently Wang-Zhang \cite{wang2013alexandroff} and
Xia-Zhang \cite{xiaabp} gave new proofs of several geometric inequalities
via the ABP method for convex domains in $\mathbb{R}^{n}$.

\section{Preliminaries}

In this section, we review fundamental concepts and properties for
the $k$-Hessian operators. For the proof of the facts below, we refer
to Garding \cite{Garding}, Reilly \cite{Reilly1973} or Guan \cite{Guan}.

%\begin{defn}
The $k$-th elementary symmetric function for $\lambda=(\lambda_{1},...,\lambda_{n})\in\mathbb{R}^{n}$ is defined as 
\[
\sigma_{k}(\lambda):=\sum_{i_{1}<i_{2}<\cdots<i_{k}}\lambda_{i_{1}}\lambda_{i_{2}}\cdots\lambda_{i_{k}},\quad1\leq k\leq n.
\]
Let $\mathcal{S}_{n}$ be the set of all symmetric $n\times n$ matrices.
The $k$-th elementary symmetric function for $A\in\mathcal{S}_{n}$
is 
\[
\s_{k}(A):=\s_{k}(\l(A)),\quad\l(A)\hbox{ is the eigenvalue of }A.
\]

The Garding cone $\Gamma_{k}^{+}$ is defined as 
\[
\Gamma_{k}^{+}=\left\{ \lambda\text{\ensuremath{\in}}\mathbb{R}^{n}|\sigma_{i}(\lambda)>0,\hbox{for}1\le i\le k\right\} .
\]
We say $A\in\mathcal{S}_{n}$ belongs to $\Gamma_{k}^{+}$ if its
eigenvalue $\l(A)\in\Gamma_{k}^{+}$. %\end{defn}We use the convention
We use the convention that $\sigma_{0}=1$.

For $A\in\mathcal{S}_{n}$, $\s_{k}(A)$ can be expressed as 
\[
\s_{k}(A)=\frac{1}{k!}\sum_{\substack{i_{1},\cdots i_{k},\\
j_{1},\cdots,j_{k}
}
}\delta_{j_{1}\cdots j_{k}}^{i_{1}\cdots i_{k}}A_{i_{1}j_{1}}\cdots A_{i_{k}j_{k}},
\]
where $\delta_{j_{1}\cdots j_{k}}^{i_{1}\cdots i_{k}}$ is the generalized
Kronecker symbol. The $k$-th Newton transformation for $A\in\mathcal{S}_{n}$
is the matrix 
\[
[T_{k}]_{ij}(A):=\frac{\p\s_{k+1}(A)}{\p A_{ij}}.
\]

It is well-known that for $A\in\Gamma_{k}^{+}$, the following Newton-Maclaurin
inequalities hold: 
\begin{eqnarray}
\left(\frac{\s_{k}(A)}{C_{n}^{k}}\right)^{\frac{1}{k}}\geq\left(\frac{\s_{l}(A)}{C_{n}^{l}}\right)^{\frac{1}{l}},\quad1\le k<l\le n.\label{NM}
\end{eqnarray}

\begin{defn}
For $A_{1},A_{2},\cdots A_{k}\in\mathcal{S}_{n}$, the polarization
of $\sigma_{k}$ is defined to be 
\[
\sigma_{k}(A_{1},\cdots,A_{k}):=\frac{1}{k!}\sum_{\substack{i_{1},\cdots i_{k},\\
j_{1},\cdots,j_{k}
}
}\delta_{j_{1}\cdots j_{k}}^{i_{1}\cdots i_{k}}(A_{1})_{i_{1}j_{1}}\cdots(A_{k})_{i_{k}j_{k}}.
\]
The mixed $k$-th Newton transformation is defined as 
\[
[T_{k}]_{ij}(A_{1},\cdots,A_{k}):=\frac{1}{k!}\sum_{\substack{i_{1},\cdots i_{k},\\
j_{1},\cdots,j_{k}
}
}\delta_{jj_{1}\cdots j_{k}}^{ii_{1}\cdots i_{k}}(A_{1})_{i_{1}j_{1}}\cdots(A_{k})_{i_{k}j_{k}}.
\]

\end{defn}
\qed %Here and in the whole paper, we sum the%repeated indices.

It is clear by definition that $\sigma_{k}(A_{1},\cdots,A_{k})$ and
$[T_{k}]_{ij}(A_{1},\cdots,A_{k})$ is multilinear with respect to
each variables. Also we have 
\[
\sigma_{k}(A)=\sigma_{k}(A,\cdots,A),\quad[T_{k}]_{ij}(A)=[T_{k}]_{ij}(A,\cdots,A),
\]
and 
\begin{equation}
(k+1)\sigma_{k+1}(A)=\sum_{i,j}A_{ij}[T_{k}]_{ij}(A).\label{xeq4}
\end{equation}

From the multilinear property, we see that for two matrices $A_{1},A_{2}\in\mathcal{S}_{n}$,
\begin{eqnarray}
\s_{k}(A_{1}+A_{2})=\sum_{l=0}^{k}C_{k}^{l}\sigma_{l}(\overbrace{A_{1},\cdots,A_{1}}\limits ^{l},A_{2},\cdots,A_{2})\label{xeq5}
\end{eqnarray}
and 
\begin{eqnarray}
[T_{k}]_{ij}(A_{1}+A_{2})=\sum_{l=0}^{k}C_{k}^{l}[T_{l}]_{ij}(\overbrace{A_{1},\cdots,A_{1}}\limits ^{l},A_{2},\cdots,A_{2}).\label{xeq6}
\end{eqnarray}

For a $C^{2}$ function $u$ on $\mathbb{R}^{n}$, we have the following 
\begin{prop}
\

\begin{itemize}
\item[(i)] If $\lambda(D^{2}u)\in\Gamma_{k}^{+}$, then $(\frac{\partial\sigma_{k}}{\partial u_{ij}}(D^{2}u))$
is positive definite, and $\sigma_{k}^{\frac{1}{k}}(D^{2}u)$ is concave
with respect to $D^{2}u$. 
\item[(ii)] $T_{k}(D^{2}u)$ is divergence free, i.e., 
\begin{equation}
\sum_{j}\partial_{j}([T_{k}]_{ij})(D^{2}u)=0.\label{divfree}
\end{equation}

\end{itemize}
\end{prop}
\

\section{Classical Neumann problems}

In this section we study \eqref{eq:classic} and prove Theorem \ref{thm:Classic}.
%\begin{equation}%\begin{cases}%\sigma_{k}(D^{2}u)=f(x), & \begin{array}{cc}%in & \Omega\end{array}\\%u_{\nu}=\lambda+\varphi(x). & \begin{array}{cc}%on & \partial\Omega\end{array}%\end{cases}\label{eq:Classic}%\end{equation}It
is clear that if $u$ is a solution then $u+c$ is also a solution
of (\ref{eq:classic}). Hence we can not expect a $C^{0}$ estimate
for $u$. We note that the $C^{1}$ and $C^{2}$ estimates proved
in \cite{ma2015neumann}, is still true for (\ref{eq:classic}). However
these estimates depend on the $C^{0}$ estimate. The main task here
is to find a gradient estimate which does not depend on the $C^{0}$
estimate. As described in the introduction, we will consider the perturbed
Neumann problem \eqref{eq:eps} and establish the following gradient
estimate. 
\begin{prop}
\label{prop} Let $\Omega$ be a smooth bounded, uniformly convex
domain in $\mathbb{R}^{n}$. Let $f\in C^{2}(\overline{\Omega})$
with $f>0$and $\varphi\in C^{3}(\overline{\Omega})$. Let $\epsilon>0$
be any positive constant. Then the Neumann problem \eqref{eq:eps}
admits a unique $k$ addmissible solution $u_{\epsilon}$. %of the following equation %%\end{thm}%\begin{equation}%\begin{cases}%\sigma_{k}(D^{2}u)=f(x), & \begin{array}{cc}%in & \Omega\end{array}\\%u_{\nu}=-\epsilon u+\varphi(x). & \begin{array}{cc}%on & \partial\Omega\end{array}%\end{cases}\label{eq:epsilon}%\end{equation}
Moreover,for sufficiant small constants $\epsilon>0$, there exists a constant
$C$, depending on $k$, $n$, $||f||_{C^{1}}$, $||\varphi||_{C^{3}}$,
and the uniform convexity of $\p\Omega$, but independent of $\epsilon$
and $||u_{\epsilon}||_{C^{0}}$ , such that%$u_\epsilon$ satisfies%estimate
\begin{equation}
\sup_{\overline{\Omega}}|\nabla u_{\epsilon}|\leq C,\label{eq:gradient}
\end{equation}
\begin{equation}
\sup_{\overline{\Omega}}\left|u_{\epsilon}-\fint_{\Omega}u_{\ep}\right|\leq C.\label{eq:gradient1}
\end{equation}
Here $\fint_{\Omega}u_{\ep}=\frac{1}{{\rm Vol}(\Omega)}\int_{\Omega}u_{\ep}$. 
\end{prop}
%here constant $C$ depents only on $k$, $n$, $||f||_{C^{1}}$, $||\varphi||_{C^{3}}$,%$\Omega$ uniformaly convexity, but not depends on $\epsilon$ and%$||u||_{C^{0}}$. Which

We give two remarks before the proof.
\begin{rem}
\

\begin{itemize}
\item[(i)] Once we have Proposition \ref{prop}, it is standard to give a Schauder
type estimate independent of $\epsilon$: 
\begin{equation}
\left\Vert u_{\ep}-\fint_{\Omega}u_{\ep}\right\Vert _{C^{2,\alpha}}\leq C.
\end{equation}
%where constant $C$ has the similar dependence.
\item[(ii)] This kind of gradient estimate relies heavily on the special structure
of (\ref{eq:eps}) and the uniform convexity of $\partial\Omega$.
For general case, a $C^{0}$ estimate of $u$ is indispensable to
get a gradient estimate. 
\end{itemize}
\end{rem}
In the following we will use the notation $F=\sigma_{k}$, $F^{ij}=\frac{\p\s_{k}}{\p u_{ij}}$
and we sum over the repeated indices.

\
 \textbf{Proof of Proposition \ref{prop}.} The existence part has
been proved in \cite{ma2015neumann}, Theorem 1.1. We now prove the
a priori estimate independent of $\epsilon$. For simplicity, we omit
the subscription for $u_{\epsilon}$.

\noindent \textit{Step 1.} We prove by maximum principle 
\[
\sup_{\overline{\Omega}}|\epsilon u|\leq C.
\]
Assuming $0\in\Omega$, we consider $u-A|x|^{2}$. There is a large
constant $A$ depending on $k$, $n$ and $\sup f$, such that 
\begin{eqnarray}
F[D^{2}u]=f\leq F[D^{2}(A|x|^{2})].\label{maxf}
\end{eqnarray}
The maximum principle applying on \eqref{maxf} yields that $u-A|x|^{2}$
attains its minimum at some boundary point $x_{0}$. So 
\begin{equation}
0\geq(u-A|x|^{2})_{\nu}(x_{0})=(-\epsilon u+\varphi-2Ax\cdot\nu)(x_{0}).\label{c0eq1}
\end{equation}
Similarly, since $u$ is a $k$ admissible solution, it is a subharmonic
function. Then $u$ attains its maximum at some boundary point $y_{0}$.
So 
\begin{equation}
0\leq u_{\nu}(y_{0})=(-\epsilon u+\varphi)(y_{0}).\label{c0eq2}
\end{equation}
It follows from \eqref{c0eq1} and \eqref{c0eq2} that 
\[
\inf\limits _{\partial\Omega}\varphi-4A\text{diam}\Omega\leq\epsilon u\leq\sup\limits _{\partial\Omega}\varphi\text{.}
\]

\noindent \textit{Step 2.} We prove the gradient estimate \eqref{eq:gradient}.

Without loss of generality, we assume $\int_{\Omega}u=0$ because
otherwise we can use $v=u-\fint_{\Omega}u$ and $\tilde{\varphi}(x)=\varphi(x)-\epsilon\fint_{\Omega}u$
instead of $u$ and $\varphi$. %$v$ satisfies%\begin{equation}%\begin{cases}%\sigma_{k}(D^{2}u)=f(x), & \begin{array}{cc}%in & \Omega\end{array}\\%u_{\nu}=-\epsilon u+\varphi(x). & \begin{array}{cc}%on & \partial\Omega\end{array}%\end{cases}%\end{equation}%Here $f$, $\varphi$ can be controled by previous $f$, $\varphi$.

%Then we prove the gradient estimate under the assumption on $\int_{\Omega}u=0$,%and other results on the theorem are consequence in \cite{ma2015neumann}.
We Consider an auxiliary function 
\begin{equation}
P=\log|Dw|^{2}+\alpha|x|^{2}\text{,}
\end{equation}
where $w=u+(-\epsilon u+\varphi)d$ and $d=d(\cdot,\p\O)$ is the
distance function from $\partial\Omega$ defined in $\{x\in\mathbb{R}^{n}|d(x,\p\O)<\mu\}$
for some small $\mu$ with smooth extension on $\overline{\Omega}$ so that $||d||_{C^{3}}$ is bounded, $\a$ is some positive constant 
 to be determined later.

Suppose $P$ attains its maximum at an interior point $z_{0}$ of
$\Omega$. At $z_{0}$, we have 
\begin{equation}
0=P_{i}=2\frac{w_{l}w_{li}}{|Dw|^{2}}+2\alpha x_{i},\label{eq:p1}
\end{equation}
and 
\begin{equation}
0\geq F^{ij}P_{ij}=\frac{2F^{ij}w_{lj}w_{li}}{|Dw|^{2}}+\frac{2F^{ij}w_{l}w_{lij}}{|Dw|^{2}}-\frac{4F^{ij}w_{l}w_{li}w_{p}w_{pj}}{|Dw|^{4}}+2\alpha\sum_{i}F^{ii}.\label{eq:P213}
\end{equation}
By the homogeneity of $F$ and taking first derivative of $F(D^{2}u)=f$,
we have 
\begin{eqnarray}
F^{ij}u_{ij} & = & kf,\label{eq:f}\\
F^{ij}u_{ijl} & = & f_{l}.\label{eq:f1}
\end{eqnarray}
Taking derivatives to function $w$, we get 
\begin{equation}
w_{i}=u_{i}+(-\epsilon u_{i}+\varphi_{i})d+(-\epsilon u+\varphi)d_{i},\label{eq:w1}
\end{equation}
\begin{equation}
w_{ij}=u_{ij}+(-\epsilon u_{ij}+\varphi_{ij})d+(-\epsilon u_{i}+\varphi_{i})d_{j}+(-\epsilon u_{j}+\varphi_{j})d_{i}+(-\epsilon u+\varphi)d_{ij},\label{eq:w2}
\end{equation}
and 
\begin{eqnarray}
w_{ijl} & = & u_{ijl}+(-\epsilon u_{ijl}+\varphi_{ijl})d+(-\epsilon u_{ij}+\varphi_{ij})d_{l}\nonumber \\
 &  & +(-\epsilon u_{il}+\varphi_{il})d_{j}+(-\epsilon u_{i}+\varphi_{i})d_{jl}+(-\epsilon u_{jl}+\varphi_{jl})d_{i}\nonumber \\
 &  & +(-\epsilon u_{j}+\varphi_{j})d_{il}+(-\epsilon u_{l}+\varphi_{l})d_{ij}+(-\epsilon u+\varphi)d_{ijl}.\label{eq:w3}
\end{eqnarray}
We choose the coordinate so that $|Dw|=w_{1}$ and $(u_{ij})_{2\leq i,j\leq n}$
is diagonal at $z_{0}$. From (\ref{eq:f}), \eqref{eq:f1} and (\ref{eq:w3}),
we have 
\begin{eqnarray}
\frac{2F^{ij}w_{l}w_{lij}}{|Dw|^{2}} & \geq & \frac{2w_{l}[(1-\epsilon)F^{ij}u_{ijl}-\epsilon F^{ij}u_{ij}d_{l}-2\epsilon F^{ij}u_{il}d_{j}]}{|Dw|^{2}}\nonumber \\
 &  & -\frac{C\sum F^{ii}(1+w_{1}+\epsilon w_{1}^{2})}{|Dw|^{2}}\nonumber \\
 & \geq & \frac{-C[\sum F^{ii}(1+w_{1}+\epsilon w_{1}^{2})+w_{1}]-4\epsilon F^{ij}u_{1i}d_{j}w_{1}}{w_{1}^{2}},\label{eq:F1}
\end{eqnarray}
here $\epsilon$ is small, such that $\epsilon d<\frac{1}{2}$, constant
$C$ depence on $||\varphi||_{C^{3}}$, $||f||_{C^{1}}$, $n$, $k$,
$||\partial\Omega||_{C^{3}}$ may be changed from line to line.

Using (\ref{eq:p1}) and (\ref{eq:w2}), we have 
\begin{eqnarray}
|u_{1i}| & \le & C(|w_{1i}|+\epsilon|Du|+1)\nonumber \\
 & \le & C(|Dw|+1).\label{eq:u1i}
\end{eqnarray}
Using (\ref{eq:u1i}), we continue to compute (\ref{eq:F1}) to get
\begin{equation}
\frac{2F^{ij}w_{l}w_{lij}}{|Dw|^{2}}\geq\frac{-C[\sum F^{ii}(1+w_{1}+\epsilon w_{1}^{2})+w_{1}]}{w_{1}^{2}}\text{,}\label{eq:F11}
\end{equation}
On the other hand, we have by \eqref{eq:p1} again that 
\begin{eqnarray}
\frac{2F^{ij}w_{lj}w_{li}}{|Dw|^{2}}-\frac{4F^{ij}w_{l}w_{li}w_{p}w_{pj}}{|Dw|^{4}} & \geq & -\frac{2F^{ij}w_{1i}w_{1j}}{w_{1}^{2}}\nonumber \\
 & = & -2\alpha^{2}F^{ij}x_{i}x_{j}.\label{eq:xeq3}
\end{eqnarray}
If we choose $\a$ and $\epsilon$ such that $\frac{1}{\max_{\O}|x|^{2}}\geq\alpha\geq2C\epsilon$
and $|Dw|$ sufficient large, we have from \eqref{eq:F11} and \eqref{eq:xeq3}
that 
\begin{eqnarray*}
F^{ij}P_{ij} & = & \frac{2F^{ij}w_{lj}w_{li}}{|Dw|^{2}}+\frac{2F^{ij}w_{l}w_{lij}}{|Dw|^{2}}-\frac{4F^{ij}w_{l}w_{li}w_{p}w_{pj}}{|Dw|^{4}}+2\alpha\sum_{i}F^{ii}\\
 & \geq & \frac{-C[\sum F^{ii}(1+w_{1}+\epsilon w_{1}^{2})+1]}{w_{1}^{2}}+2\alpha\sum F^{ii}-2\alpha^{2}F^{ij}x_{i}x_{j}\\
 & > & 0.
\end{eqnarray*}
This contradicts with (\ref{eq:P213}). Thus $P$ can only attain
its maximum at boundary points.

Suppose $P$ attains its maximum on a boundary point $\tilde{z}_{0}$.
Choose a local orthonormal frame $\{\partial_{i}\}_{i=1}^{n}$ so
that $\p_{n}=\nu$. From the boundary condition, we have $w_{n}=0$
at $\tilde{z}_{0}$. By the maximal property of $P$ at $\tilde{z}_{0}$,
we have 
\begin{eqnarray}
0\leq P_{\nu} & = & 2\frac{w_{l}w_{l\nu}}{|Dw|^{2}}+2\alpha x\cdot\nu=2\frac{\sum\limits _{\a=1}^{n-1}w_{\a}w_{\a n}}{|Dw|^{2}}+2\alpha x\cdot\nu.\label{eq:Pnu}
\end{eqnarray}

By taking the tangential derivative for the boundary condition along
$\p\O$, we have 
\begin{equation}
u_{\a n}+\sum\limits _{\b=1}^{n-1}h_{\a\b}u_{\b}=-\epsilon u_{\a}+\varphi_{\a},\quad\a=1,\cdots,n-1.\label{xeq1}
\end{equation}
here $h_{\a\b}$ is the second fundamental form of $\p\O$. The equation
(\ref{eq:w1}) tells us ()
\[
(1-\epsilon d)u_{\a}-C\leq w_{\a}\leq(1-\epsilon d)u_{\a}+C,
\]
which also infer that 
\begin{equation}
\frac{1}{4}\sum\limits _{\a=1}^{n-1}u_{\a}^{2}-C\leq|Dw|^{2}\leq4\sum\limits _{i=1}^{n-1}u_{\a}^{2}+C.\label{xeq2}
\end{equation}
Since $\Omega$ is uniformly convex, $h_{\a\b}\geq c_{0}\d_{\a\b}$
for some $c_{0}>0$. Thus we deduce from \eqref{xeq1} and \eqref{xeq2}
that 
\begin{eqnarray*}
\sum\limits _{\a=1}^{n-1}w_{\a}w_{\a n} & \leq & -\frac{c_{0}}{4}\sum\limits _{\a=1}^{n-1}u_{\a}^{2}-\epsilon|Dw|^{2}-C|Dw|-C\\
 & \leq & -\frac{c_{0}}{16}|Dw|^{2}-\epsilon|Dw|^{2}-C|Dw|-C.
\end{eqnarray*}
By choosing $\epsilon\leq\frac{c_{0}}{32}$,$\alpha\leq\frac{c_{0}}{32\max x\cdot\nu}$
and $|Dw|$ sufficient large, we get 
\begin{eqnarray*}
P_{\nu}\leq-\frac{c_{0}}{16}+2\alpha x\cdot\nu\leq0.
\end{eqnarray*}
This is a contradiction to (\ref{eq:Pnu}).

In conclusion we first choose $\alpha$ small and then $\epsilon$
small, we get two contradictions. That means, we cannot have $|Dw|$
large. Hence we get the upper bound of $|Dw|$ and in turn also $|Du|$.

\noindent \textit{Step 3.} The a priori estimate \eqref{eq:gradient1}
follows from the gradient estimate by the Poincare inequality. We
complete the proof of Proposition \ref{prop}. \qed

\
 Now we readily prove Theorem \ref{thm:Classic}. 
\begin{proof}
Let $u_{\epsilon}$ be a solution of (\ref{eq:eps}) for any $\ep>0$.
Because $|\nabla(-\epsilon u)|\rightarrow0$ and the Schauder estimate,
there is a constant $\lambda$ and a function $\bar{u}\in C^{2}(\bar{\O})$,
such that 
\[
-\epsilon u\rightarrow\lambda,\hbox{and}\quad u_{\epsilon}-\fint_{\O}u_{\epsilon}\to\bar{u}\hbox{ uniformly in }C^{2}\hbox{ as }\epsilon\to0.
\]
It follows that $\bar{u}$ solves the following classical Neumann
problem 
\begin{equation}
\begin{cases}
\sigma_{k}(D^{2}u)=f(x) & \begin{array}{cc}
 & \hbox{ in }\Omega,\end{array}\\
u_{\nu}=\lambda+\varphi(x) & \begin{array}{cc}
 & \hbox{ on }\partial\Omega.\end{array}
\end{cases}\label{eq:Classic-1}
\end{equation}

Next we prove the uniqueness. Suppose problem (\ref{eq:Classic-1})
has two pairs of solutions $(\lambda,u)$ and $(\mu,v)$. Let $a^{ij}=\int_{0}^{1}F^{ij}[(1-t)D^{2}v+tD^{2}u]dt$,
and $u-v$ satisfies

\[
\begin{cases}
a^{ij}(u-v)_{ij}=0,\\
(u-v)_{\nu}=\lambda-\mu.
\end{cases}
\]
It follows that $u-v$ attains its maximum and its minimum both at
some boundary points. It shows that $\lambda=\mu$. Finally, the Hopf
lemma in \cite[Theorem 3.6]{gilbarg2015elliptic} yields $u-v=c$. 
\end{proof}
\

\section{Alexandrov-Fenchel Inequalities}

In this chapter, we use the solution of the classical Neumann problems
to give a new proof of some Alexandrov-Fenchel inequalities. We need
the following result due to Reilly \cite{reilly1977applications}. 
\begin{prop}
\label{Reilly} Let $\Omega$ be a smooth, bounded convex domain in
$\mathbb{R}^{n}$. Let $u\in C^{2}(\bar{\Omega})$ such that $\lambda(D^{2}u)\in\Gamma_{k}^{+}$
and $u_{\nu}=c$ on $\p\O$, where $c$ is a positive constant, then
the following inequality holds: 
\begin{equation}
(k+1)\int_{\Omega}\sigma_{k+1}(D^{2}u)\geq\int_{\partial\Omega}\sigma_{k}(h)c^{k+1}.
\end{equation}
\end{prop}
\begin{proof}
For completeness, we give a proof here.

In the following we choose an orthonormal coordinate $\{\p_{i}\}_{i=1}^{n}$
such that $\p_{n}=\nu$ on the boundary. We denote by $D_{ij}^{2}u$
the Hessian with respect to the ambient (Euclidean) metric, $u_{\a\b}=\n_{\a\b}^{2}u$
the Hessian of $u$ with respect to the induced metric on $\p\O$
and $u_{n\alpha}=\nabla_{\alpha}(u_{n})$ on $\p\O$. We will sum
the repeated indices as before and also the convention that the Latin
indices run through $1$ to $n$ while the Greek indices run through
$1$ to $n-1$.

%One can write the Hessian of $u$ in coordinates of tangential derivatives%and normal derivatives. Indices $i,j,k$ ranging from $1$ to $n$%be the coordinates of $\mathbb{R}^{n}$, indices $\alpha,\beta,\gamma$%ranging from $1$ to $n-1$ be the tangential directions, and $\partial_{n}$%be the outward unit normal direction on $\partial\Omega$. $D_{\alpha\beta}^{2}u$%means derivative with respect to the ambient (Euclidean) metric, $u_{\alpha\beta}$%with respect to the metric of the surface measure on $\partial\Omega$.%And we denote $\nabla_{\alpha}(u_{n})$ by $u_{n\alpha}$. Let $h_{\alpha\beta}$%be the second fundamental form at $x\in\partial\Omega$. It is known%that

It follows from the Gauss-Weingarten formula that 
\begin{eqnarray}
D_{\alpha\beta}^{2}u&=&u_{\alpha\beta}+h_{\alpha\beta}u_{n},  \label{eq:xeq9}\\
D_{\alpha n}^{2}u&=&u_{n\alpha}-h_{\alpha\beta}u_{\beta}.
\end{eqnarray}

Set $A$ and $B$ to be the following two matrices: 
\[
A:=\left(\begin{matrix}\cdots & \cdots & \cdots & \vdots\\
\cdots & u_{\alpha\beta} & \cdots & u_{n\alpha}\\
\cdots & \cdots & \cdots & \vdots\\
\cdots & u_{n\alpha} & \cdots & u_{nn}
\end{matrix}\right),
\]
\[
B:=\left(\begin{matrix}\cdots & \cdots & \cdots & \vdots\\
\cdots & h_{\alpha\beta}u_{n} & \cdots & -h_{\alpha\gamma}u_{\gamma}\\
\cdots & \cdots & \cdots & \vdots\\
\cdots & -h_{\beta\gamma}u_{\gamma} & \cdots & 0
\end{matrix}\right).
\]
Then on $\p\O$, $D^{2}u$ is decomposed to be 
\[
D^{2}u=A+B.
\]

By using \eqref{xeq4} and \eqref{divfree} we see 
\begin{eqnarray*}
(k+1)\int_{\Omega}\sigma_{k+1}(D^{2}u)=\int_{\partial\Omega}[T_{k}]_{in}(D^{2}u)D_{ij}^{2}u=\int_{\partial\Omega}[T_{k}]_{in}(D^{2}u)u_{i}%
=I+II,
\end{eqnarray*}
where 
\[
I:=\int_{\partial\Omega}[T_{k}]_{nn}(D^{2}u)u_{n}\text{,}\quad\hbox{and}\quad II:=\int_{\partial\Omega}[T_{k}]_{\alpha n}(D^{2}u)u_{\alpha}.
\]

Note that 
\begin{eqnarray*}
[T_{k}]_{nn}(D^{2}u)=\frac{1}{k!}\sum_{\substack{\a_{1},\cdots\a_{k},\\
\b_{1},\cdots,\b_{k}
}
}\delta_{\a_{1}\cdots\a_{k}}^{\b_{1}\cdots\b_{k}}D_{\a_{1}\b_{1}}^{2}u\cdots D_{\a_{k}\b_{k}}^{2}u=\s_{k}(D^{2}u|_{\p\O})
\end{eqnarray*}
and 
\begin{eqnarray*}
[T_{k}]_{\a n}(D^{2}u) & = & -\frac{1}{(k-1)!}\sum_{\substack{\a_{1},\cdots\a_{k-1},\\
\b_{1},\cdots,\b_{k-1}\neq\a
}
}\delta_{\a\a_{2}\cdots\a_{k-1}}^{\b_{1}\b_{2}\cdots\b_{k-1}}D_{\a_{2}\b_{2}}^{2}u\cdots D_{\a_{k-1}\b_{k-1}}^{2}uD_{\b_{1}n}^{2}u\\
 & = & [T_{k-1}]_{\alpha\beta_{1}}(D^{2}u|_{\p\O})D_{\b_{1}n}^{2}u.
\end{eqnarray*}

Using the assumption $u_{n}=c$, we know $\n_{\b_{1}}u_{n}=0$ and
hence $D_{\b_{1}n}^{2}u=-h_{\b_{1}\a_{1}}u_{\a_{1}}$. By using the
decomposition $D^{2}u=A+B$ and \eqref{xeq6}, we have 
\begin{eqnarray*}
II & = & \int_{\partial\Omega}[T_{k-1}]_{\alpha\beta_{1}}(D^{2}u|_{\p\O})h_{\alpha_{1}\beta_{1}}u_{\alpha_{1}}u_{\alpha}\\
 & = & \int_{\partial\Omega}\sum\limits _{l=0}^{k-1}C_{k-1}^{l}[T_{k-1}]_{\alpha\beta_{1}}(\overbrace{\nabla^{2}u,\cdots,\nabla^{2}u}\limits ^{l},h,\cdots,h)u_{n}^{k-1-l}h_{\alpha_{1}\beta_{1}}u_{\alpha_{1}}u_{\alpha}.
\end{eqnarray*}
Similarly, we have 
\begin{eqnarray*}
I & = & \int_{\partial\Omega}\sum\limits _{l=1}^{k}C_{k}^{l}\sigma_{k}(\overbrace{\nabla^{2}u,\cdots,\nabla^{2}u}\limits ^{l},h,\cdots,h)u_{n}^{k-l+1}+\int_{\partial\Omega}\sigma_{k}(h)u_{n}^{k+1}\\
 & =: & I_{1}+I_{2}.
\end{eqnarray*}

The term $I_{2}$ is what we want in final. We compute $I_{1}$ further.
By intergration by parts and using again $u_{n}=c$, we get 
\begin{eqnarray}
I_{1} & = & \int_{\partial\Omega}\sum\limits _{l=1}^{k}C_{k}^{l}\sigma_{k}(\overbrace{\nabla^{2}u,\cdots,\nabla^{2}u}\limits ^{l},h,\cdots,h)u_{n}^{k-l+1}\label{xeq8}\\
 & =- & \int_{\partial\Omega}\sum_{l=2}^{k}C_{k}^{l}\frac{l-1}{k!}\delta_{\beta_{1}\cdots\beta_{k}}^{\alpha_{1}\cdots\alpha_{k}}u_{\alpha_{1}}\cdots u_{\alpha_{l}\beta_{l}\beta_{1}}h_{\alpha_{l+1}\beta_{l+1}}\cdots h_{\alpha_{k}\beta_{k}}u_{n}^{k-l+1},\nonumber 
\end{eqnarray}
where in the second equality we also used the fact $h_{\alpha_{k}\beta_{k}\beta_{1}}=h_{\alpha_{k}\beta_{1}\beta_{k}}$,
which is the Codazzi property of the second fundamental form $h_{\a\b}$.

The Ricci identity tells that 
\[
u_{\alpha_{l}\beta_{l}\beta_{1}}-u_{\alpha_{l}\beta_{1}\beta_{l}}=u_{\gamma}R_{\gamma\alpha_{l}\beta_{l}\beta_{1}}=u_{\gamma}h_{\gamma\beta_{l}}h_{\alpha_{l}\beta_{1}}-u_{\gamma}h_{\gamma\beta_{1}}h_{\alpha_{l}\beta_{l}}.
\]

Replacing the above into \eqref{xeq8}, we have 
\begin{eqnarray*}
I_{1} & = & \int_{\partial\Omega}\sum_{l=2}^{k}C_{k}^{l}\frac{l-1}{k!}\delta_{\beta_{1}\cdots\beta_{k}}^{\alpha_{1}\cdots\alpha_{k}}u_{\alpha_{1}}u_{\alpha_{2}\beta_{2}}\cdots u_{\alpha_{l-1}\beta_{l-1}}u_{\gamma}h_{\gamma\beta_{1}}h_{\alpha_{l}\beta_{l}}h_{\alpha_{l+1}\beta_{l+1}}\cdots h_{\alpha_{k}\beta_{k}}u_{n}^{k-l+1},\\
 & = & \int_{\partial\Omega}\sum_{l=2}^{k}C_{k}^{l}\frac{l-1}{k}[T_{k-1}]_{\alpha_{1}\beta_{1}}(\overbrace{\nabla^{2}u,\cdots,\nabla^{2}u}^{l-2},h,\cdots,h)u_{\alpha_{1}}u_{\gamma}h_{\gamma\beta_{1}}u_{n}^{k-l+1}\\
 & = & \int_{\partial\Omega}\sum_{l=0}^{k-1}C_{k}^{l+2}\frac{l+1}{k}[T_{k-1}]_{\alpha_{1}\beta_{1}}(\overbrace{\nabla^{2}u,\cdots,\nabla^{2}u}^{l+1},h,\cdots,h)u_{\alpha_{1}}u_{\gamma}h_{\gamma\beta_{1}}u_{n}^{k-l-1}.
\end{eqnarray*}

We sum $II$ and $I_{1}$ to be $III$.

\begin{eqnarray}
III & := & I_{1}+II\nonumber \\
 & = & \int_{\partial\Omega}\sum_{l=0}^{k-1}C_{k+1}^{l+2}\frac{l+1}{k}[T_{k-1}]_{\alpha_{1}\beta_{1}}(\overbrace{\nabla^{2}u,\cdots,\nabla^{2}u}^{l},h,\cdots,h)u_{\alpha_{1}}u_{\gamma}h_{\gamma\beta_{1}}u_{n}^{k-l-1}.\nonumber \\
\end{eqnarray}

In order to see the sign of $III$, we use $\nabla_{\alpha\beta}^{2}u=D_{\alpha\beta}^{2}u-h_{\alpha\beta}u_{n}$
in \eqref{eq:xeq9} to get 
\begin{eqnarray*}
III & = & \int_{\partial\Omega}\sum_{l=0}^{k-1}\sum_{i=0}^{l}(-1)^{l-i}C_{k+1}^{l+2}C_{l}^{i}\frac{l+1}{k}[T_{k-1}]_{\alpha_{1}\beta_{1}}(\overbrace{D^{2}u,\cdots,D^{2}u}^{i},h,\cdots,h)u_{\alpha_{1}}u_{\gamma}h_{\gamma\beta_{1}}u_{n}^{k-i-1}\\
 & = & \int_{\partial\Omega}\sum_{i=0}^{k-1}\sum_{l=i}^{k-1}(-1)^{l-i}C_{k+1}^{l+2}C_{l}^{i}\frac{l+1}{k}[T_{k-1}]_{\alpha_{1}\beta_{1}}(\overbrace{D^{2}u,\cdots,D^{2}u}^{i},h,\cdots,h)u_{\alpha_{1}}u_{\gamma}h_{\gamma\beta_{1}}u_{n}^{k-i-1}.
\end{eqnarray*}

If we can prove that $III\geq0$, we are done. We see first that since
$\lambda(D^{2}u)\in\Gamma_{k}^{+}$ and $h_{\alpha\beta}$ is nonnegative,
\begin{equation}
[T_{k-1}]_{\alpha_{1}\beta_{1}}(\overbrace{D^{2}u,\cdots,D^{2}u}^{i},h,\cdots,h)h_{\gamma\beta_{1}}u_{\alpha_{1}}u_{\gamma}\geq0
\end{equation}
by the Garding inequality (See \cite{Garding}).

Hence to prove $III\geq0$, we need to only to prove 
\begin{equation}
E:=\sum_{l=i}^{k-1}(-1)^{l-i}C_{k+1}^{l+2}C_{l}^{i}\frac{l+1}{k}\text{\ensuremath{\geq}0}.
\end{equation}
Following a trick in \cite{qiu2015family}, we deal this term as follows:
\begin{eqnarray*}
E & = & (k+1)C_{k-1}^{i}\sum_{l=i}^{k-1}\frac{1}{l+2}C_{k-1-i}^{l-i}(-1)^{l-i}\\
 & = & (k+1)C_{k-1}^{i}\sum_{p=0}^{k-1-i}\frac{1}{p+i+2}C_{k-1-i}^{p}(-1)^{p}.
\end{eqnarray*}
Note that the folowing elementary equality holds: 
\begin{eqnarray*}
\int_{0}^{1}t^{i+1}(1-t)^{k-i-1}dt & = & \int_{0}^{1}t^{1+i}\sum_{p=0}^{k-i-1}C_{k-i-1}^{p}(-1)^{p}t^{p}dt\\
 & = & \sum_{p=0}^{k-i-1}C_{k-i-1}^{p}(-1)^{p}\frac{t^{p+2+i}}{p+2+i}\big|_{t=0}^{t=1}.
\end{eqnarray*}
Hence 
\begin{eqnarray*}
E=(k+1)C_{k-1}^{i}\int_{0}^{1}t^{1+i}(1-t)^{k-1-i}dt=\frac{i+1}{k}.
\end{eqnarray*}
Finally, we conclude 
\begin{eqnarray*}
 &  & \int_{\partial\Omega}[T_{k}]_{ij}(D^{2}u)u_{i}\nu_{j}d\mu\\
 & = & \int_{\partial\Omega}\sigma_{k}(h)u_{n}^{k+1}+\sum_{i=0}^{k-1}\frac{i+1}{k}[T_{k-1}]_{\alpha_{1}\beta_{1}}(\overbrace{D^{2}u,\cdots,D^{2}u}^{i},h,\cdots,h)h_{\gamma\beta_{1}}u_{\alpha_{1}}u_{\gamma}u_{n}^{k-i-1}\text{}\\
 & \geq & \int_{\partial\Omega}\sigma_{k}(h)u_{n}^{k+1}.
\end{eqnarray*}
The proof of Proposition \ref{Reilly} is completed. 
\end{proof}
\noindent \textbf{Proof of Theorem \ref{thm:AF}.} From Theorem \ref{thm:Classic},
there is a $k$ admissible solution $u$ and an unique constant $c$
satisfying 
\begin{equation}
\begin{cases}
\sigma_{k}(D^{2}u)=C_{n}^{k} & \begin{array}{cc}
 & \hbox{ in }\Omega,\end{array}\\
u_{\nu}=c & \begin{array}{cc}
 & \hbox{ on }\partial\Omega.\end{array}
\end{cases}\label{eq:Cla}
\end{equation}

On one hand, since $\lambda(D^{2}u)\in\Gamma_{k}^{+}$, the Newton-Maclaurin
inequality \eqref{NM} %\begin{equation}%\frac{\sigma_{k+1}}{C_{n}^{k+1}}\leq(\frac{\sigma_{k}}{C_{n}^{k}})^{\frac{k+1}{k}},%\end{equation}infers
that 
\begin{equation}
\int_{\Omega}(k+1)\sigma_{k+1}(D^{2}u)\leq\int_{\Omega}(k+1)C_{n}^{k+1}\left(\frac{\sigma_{k}(D^{2}u)}{C_{n}^{k}}\right)^{\frac{k+1}{k}}.\label{NM1}
\end{equation}
Using Proposition \ref{Reilly}, \eqref{NM1} and equation (\ref{eq:Cla}),
we have 
\begin{equation}
c^{k+1}\int_{\partial\Omega}\sigma_{k}(h)\leq(k+1)C_{n}^{k+1}{\rm Vol}(\Omega).\label{eq:first}
\end{equation}

On the other hand the Newton-MacLaurin inequality \eqref{NM} also
%\begin{equation}%(\frac{\sigma_{k}}{C_{n}^{k}})^{\frac{1}{k}}\leq\frac{\sigma_{1}}{C_{n}^{1}},%\end{equation}infers
that 
\begin{equation}
\int_{\Omega}\left(\frac{\sigma_{k}(D^{2}u)}{C_{n}^{k}}\right)^{\frac{1}{k}}\leq\int_{\Omega}\frac{\sigma_{1}(D^{2}u)}{n}=\int_{\partial\Omega}\frac{u_{\nu}}{n}=c\frac{{\rm Area}(\partial\Omega)}{n}.\label{eq:second}
\end{equation}

Combine (\ref{eq:first}) and (\ref{eq:second}), we conclude 
\[
{\rm Vol}(\Omega)^{k}\int_{\partial\Omega}\sigma_{k}(h)\leq\frac{C_{n-1}^{k}}{n^{k}}{\rm Area}(\partial\Omega)^{k+1}.
\]

When equality attains in the above inequality, from the proof in Proposition
\ref{Reilly}, we know that the tangential derivative of $u$ vanishes.
Also we see from \eqref{eq:first} that $\left(\frac{\sigma_{k}(D^{2}u)}{C_{n}^{k}}\right)^{\frac{1}{k}}=\frac{\sigma_{1}(D^{2}u)}{n}$.
So $u$ satisfies the overdetermined problem 
\[
\begin{cases}
\Delta u=n & \begin{array}{cc}
\hbox{ in } & \Omega\end{array}\\
u_{\nu}=c & \begin{array}{cc}
\hbox{ on } & \partial\Omega\end{array}\\
u=b & \begin{array}{cc}
\hbox{ on } & \partial\Omega,\end{array}
\end{cases}
\]
where $b$ is a constant.

By using the result of Serrin \cite{serrin1971symmetry}, we have
that \eqref{eq:AFIne} becomes equality if and only $\Omega$ is a
ball. %, moreover $u=\frac{|x|^{2}}{2}$%and $b=\frac{c^{2}}{2}$.
 The proof is completed. \qed

\

\noindent  \textbf{Acknowledgement. }
 The first author would like to express  gratitude to Prof. Xinan
Ma for his advice and constant encouragement. Both authors would
like to thank Prof. Pengfei Guan for valuable discussions and support.

\

\end{document}